\documentclass[12pt]{amsart}

\usepackage{amsmath}
\usepackage{amssymb}
\usepackage{amsthm}
\usepackage{xcolor}
\usepackage{bm}
\usepackage{graphicx}
\usepackage{pifont}
\usepackage{mathrsfs}
\usepackage{epsfig,verbatim}
\usepackage[pagewise]{lineno}\linenumbers

\newcommand\be{\begin{equation}}
\newcommand\ee{\end{equation}}
\newcommand\bea{\begin{eqnarray}}
\newcommand\eea{\end{eqnarray}}
\newcommand\beaa{\begin{eqnarray*}}
\newcommand\eeaa{\end{eqnarray*}}
\newcommand\beba{\begin{equation}\left\{\begin{array}{rcl}}
\newcommand\eeba{\end{array}\right.\end{equation}}
\newcommand\bebaa{\begin{equation*}\left\{\begin{array}{rcl}}
\newcommand\eebaa{\end{array}\right.\end{equation*}}

\newcommand\R{\mathbb{R}}
\newcommand\bR{{\mathbb{R}}}
\newcommand\bC{{\mathbb{C}}}
\newcommand\bN{{\mathbb{N}}}

\newcommand\cE{{\mathcal{E}}}

\newcommand\cL{{\mathcal{L}}}

\newcommand\cK{{\mathcal{K}}}

\newcommand{\eps}{\epsilon}

\newcommand{\dsum}{\displaystyle \sum}

\newcommand{\lng}{\left\langle  \right.}
\newcommand{\rng}{\left.  \right\rangle}
\newtheorem{theorem}{Theorem}[section]
\newtheorem{lemma}[theorem]{Lemma}
\newtheorem{corollary}[theorem]{Corollary}

\newtheorem{remark}[theorem]{Remark}
\newtheorem{proposition}[theorem]{Proposition}

\numberwithin{equation}{section}
\def\theequation{\arabic{section}.\arabic{equation}}

\begin{document}
\nolinenumbers

\title[PDE approximation]{On the approximation of spatial convolutions\\ by PDE systems}

\author[H. Ishii]{Hiroshi Ishii}
\address{Research Center of Mathematics for Social Creativity, Research Institute for Electronic Science, Hokkaido University, Hokkaido, 060-0812, Japan}
\email{hiroshi.ishii@es.hokudai.ac.jp}

\author[Y. Tanaka]{Yoshitaro Tanaka}
\address{Department of Complex and Intelligent Systems, School of Systems Information Science, Future University Hakodate, 116-2 Kamedanakano-cho, Hakodate, Hokkaido 041-8655, Japan}
\email{yoshitaro.tanaka@gmail.com}

\thanks{Date: \today. Corresponding author: Hiroshi Ishii}


\begin{abstract}
This paper develops an approximation framework for spatial convolutions with radial kernels using elliptic PDE systems in arbitrary spatial dimensions.
We prove that radial kernels in Sobolev spaces can be approximated arbitrarily well by linear sums of Green functions associated with a prescribed sequence of positive parameters. 
The proof is based on a completeness result for an auxiliary system of linear sums of Green functions, established through the identity theorem for analytic functions and a uniqueness argument by integral transforms.
The Sobolev approximation provides control of derivatives of the kernels and yields corresponding approximations of convolution operators. 
As an application, we show that solutions of an aggregation–confinement–diffusion equation can be approximated on any finite time interval by solutions of a corresponding local parabolic–elliptic Keller-Segel system. 
A numerical example illustrates the kernel approximation.
\end{abstract}

\maketitle
\setcounter{tocdepth}{3}

\section{Introduction}
The temporal evolution of phenomena such as cellular motion, collective particle dynamics, and neural activity is governed by interactions among spatially distributed components. 
For example, self-organized particle motion is often modeled in terms of an interaction potential that acts over a spatially extended region \cite{CCS2018}. 
In neural systems, neurons communicate with distant cells by transmitting electrical signals through their axons \cite{Amari}. 
Such long-range interactions are commonly referred to as nonlocal interactions.
In mathematical models, they are often represented by spatial convolutions of an integral kernel with an unknown function.
In this paper, we consider the spatial convolution
\beaa
(K*u)(t,x) := \int_{\bR^n} K(x-y)u(t,y)dy = \int_{\bR^n} K(y)u(t,x-y)dy,
\eeaa
where the integral kernel $K:\bR^n\to\bR$ is a radial function 
and $u=u(t,x)$ is a function depending on the time variable $t\ge0$ and the spatial variable $x\in\bR^n$.

Evolution equations involving spatial convolutions arise in mathematical models in a wide range of scientific fields. 
In materials science, several models of phase separation formulate the internal energy of the Helmholtz free-energy functional in terms of nonlocal interactions between states \cite{Bates}. 
The resulting evolution equations incorporate nonlocal interactions and provide natural generalizations of the Allen-Cahn and Cahn-Hilliard equations.
In neuroscience, convolution-based models have been introduced to describe the neuronal interactions underlying the propagation of activity in neural fields \cite{Amari, Murray}.

Nonlocal interactions also play an important role in models of biological pattern formation. 
For example, a simple mathematical model has been proposed to capture the essential interactions among pigment cells observed experimentally in zebrafish \cite{Kondo}. 
Moreover, methods have been developed to analyze interaction structures involving multiple species within a reaction-diffusion framework and to derive effective integral kernels governing pattern formation \cite{EIKMT}. 
In cell sorting models, cell-cell adhesion has been represented through nonlocal interactions, allowing observed dynamics to be reproduced with high fidelity \cite{CMSTT}. 
Related approaches have also been used in ecology to describe spatial dispersal processes, including organism movement and seed dispersal \cite{HMMV, MRS}.

In parallel with the development of several models, a variety of analytical approaches have been established to investigate the asymptotic behavior of solutions \cite{AMRT}, the existence and stability of localized patterns \cite{BCC, BFRW, X.C, EGIW}, and weak interactions between such patterns \cite{EI}. 
Although methods that directly address nonlocal interactions have been developed, spatial convolutions continue to pose substantial challenges for both analysis and computation. 
In particular, their nonlocal property makes it difficult to characterize the local behavior of solutions and often leads to considerable computational cost in numerical simulations.

To overcome these difficulties, spatial convolutions have been approximated by partial differential operators. 
In \cite{BCC, Murray}, the convolution operator is formally replaced by a differential operator through a Taylor expansion. 
The validity of this approach was investigated in \cite{BCC} for the one-dimensional case. 
The method was shown to be applicable under restrictive conditions, for example, when the support of the integral kernel is sufficiently narrow.

Another widely used approach is to rewrite a spatial convolution as a PDE system by representing the integral kernel as a linear sum of Green functions \cite{AFSS, Faye, FH, LT}.
In particular, for $d>0$, the Green function $k(x;d)$ defined as the solution of
\beaa
d\Delta k(x) -k(x) + \delta(x) =0 \quad (x\in\bR^n)
\eeaa
is often employed, where $\Delta$ denotes the Laplace operator and $\delta(x)$ is the Dirac delta function. 
The explicit form of $k(x;d)$ is given in Section \ref{sec:3}. 
Suppose that the integral kernel is given by
\beaa
K_{M}(x) := \dsum^{M}_{m=1} a_m k_m(x),\quad k_m(x):=k(x;d_m)
\eeaa
with $M\in\bN$, constants $\{a_{m}\}_{1\le m\le M}$, and positive constants $\{d_m\}_{1\le m\le M}$.
Then the spatial convolution can be represented as
\beaa
&&(K_{M}*u)(t,x)= \dsum^{M}_{m=1} a_{m} v_m(t,x),\\
&&d_m \Delta v_m(t,x) -v_m(t,x) + u(t,x)=0\quad (m=1,2,\ldots,M)
\eeaa
by setting $v_m(t,x):=(k_m*u)(t,x)\ (m=1,\ldots,M)$.
Thus, the spatial convolution is expressed as a linear sum of solutions to the PDE system.
We refer to this system of $M$ elliptic equations for $v_1,\ldots,v_M$ as the elliptic PDE system associated with $K_M$.

More recently, these ideas have been developed into an approximation framework in which solutions of scalar nonlocal evolution equations are approximated by the $u$-components of solutions to corresponding local parabolic–elliptic PDE systems.
We consider an equation of the form
\be\label{eq:single}
u_t=\cL(x,u,K*u).
\ee
Typical examples include the nonlocal reaction-diffusion equation (e.g. \cite{EI, NTY1, NTY2})
\beaa
    u_t=\Delta u+f(u,K*u),
\eeaa
where $f:\bR^2\to\bR$ is a nonlinear function, and the aggregation-confinement--diffusion equation (e.g. \cite{CCY, GC})
\beaa
    u_t = \Delta u - \nabla\cdot \bigl[ u\bigl(\nabla V(x)+\nabla K*u\bigr) \bigr],
\eeaa
where $V:\bR^n\to\bR$ is a confinement potential.

Based on the previous discussion, if the integral kernel can be approximated by a linear sum of Green functions, then the nonlocal equation is expected to be approximated by the parabolic-elliptic PDE system
\be\label{eq:localPDE}
\begin{cases}
\displaystyle
u_t = \cL\left( x,u,\dsum_{m=1}^{M}a_mv_m \right),\\[1mm] 
0=d_m\Delta v_m-v_m+u,
\qquad (m=1,2,\ldots,M).
\end{cases}
\ee
More precisely, under suitable assumptions, the solution of \eqref{eq:single} is expected to be approximated by the corresponding solution of \eqref{eq:localPDE}.

Ninomiya et al. \cite{NTY1, NTY2} studied a one-dimensional nonlocal reaction-diffusion equation.
They showed that the solution of the nonlocal equation remains sufficiently close to that of the corresponding local PDE system, provided that the integral kernel $K$ is approximated by a linear sum $K_M$ of Green functions in suitable norms. 
Murakawa and Tanaka \cite{MT} considered a one-dimensional aggregation-confinement-diffusion equation with $V\equiv0$ and established an analogous approximation result. 
These studies also identified sufficient classes of integral kernels that can be approximated by linear sums of Green functions. 
Furthermore, they addressed the parabolic-parabolic case and investigated, for the respective choices of $\cL$, the correspondence between the dynamics of \eqref{eq:single} and those of the reaction-diffusion and Keller-Segel systems.

More recently, in \cite{ITpre}, the authors extended this approximation framework to higher-dimensional Euclidean spaces for the nonlocal reaction-diffusion equation. 
In particular, it was shown that every radial kernel $K$ in $L^1(\bR^n)$ can be approximated in the $L^1$-norm by linear sums $K_M$ of Green functions in arbitrary spatial dimensions. 
Moreover, when $n\le 3$, a constructive approximation was obtained in which the parameters of the corresponding PDE system can be determined explicitly from the kernel. 
This result provides a general $L^1$-approximation framework for nonlocal interactions in higher-dimensional spaces.

The $L^1$-approximation, however, does not in general provide control of derivatives of the approximating kernels. 
Such control becomes essential when the spatial convolution appears through its derivatives, as in aggregation-confinement-diffusion equations involving the drift term $\nabla K*u$. 
This motivates the development of an approximation theory in Sobolev spaces. 
In the present paper, we show that radial kernels can be approximated in $H^\alpha(\bR^n)$ by linear sums of Green functions in arbitrary spatial dimensions. 
As a consequence, derivatives of the kernels can also be approximated through Sobolev embedding, which enables us to establish a PDE system approximation for an aggregation--confinement--diffusion equation in higher dimensions.

The essential problem of the elliptic PDE system approximation approach is to approximate a given radial integral kernel $K$ by
\beaa
    K_M=\sum_{m=1}^{M}a_mk_m
\eeaa
through an appropriate choice of $M$, $\{a_m\}_{1\leq m\leq M}$, and $\{d_m\}_{1\leq m\leq M}$.
Approximation results in the $L^\infty$-norm have previously been established in one-dimensional bounded domains with periodic boundary conditions \cite{MT, NTY1}, 
while an $L^2$-approximation result has been obtained on the whole line \cite{NTY2}.
More recently, an $L^1$-approximation theorem for radial kernels in arbitrary spatial dimensions was established in \cite{ITpre}.
To extend PDE system approximations to equations involving derivatives of spatial convolutions, it is necessary to develop an approximation theory that also controls derivatives of the kernels.
In addition, we ask whether such an approximation can be achieved using a prescribed sequence of Green functions, rather than choosing the diffusion parameters specifically for each target kernel.

The main contributions of this paper are threefold. 
First, we prove that every radial kernel $K$ in $H^\alpha(\bR^n)$ can be arbitrarily well approximated in the $H^\alpha$-norm by a finite linear sum $K_M$ of Green functions. 
This result applies to arbitrary spatial dimensions, even when Green functions do not belong to $H^\alpha(\bR^n)$, because their singularities can be canceled by a suitable choice of coefficients. 
Second, we construct an auxiliary system of functions $\{\phi_m\}_{m\in\bN}$, each of which is a finite linear sum of Green functions, and prove that its linear span is dense in $H^\alpha_r(\bR^n)$ (see Section~\ref{sec:3} for the details).
Third, we apply the approximation result to show that solutions of an aggregation–confinement–diffusion equation can be approximated, on any finite time interval, by solutions of a corresponding local parabolic–elliptic Keller-Segel system.

The remainder of this paper is organized as follows. 
In Section~\ref{sec:2}, we state the main results concerning the approximation of radial integral kernels and spatial convolutions. 
The proofs of these results are presented in Section~\ref{sec:3}. 
In Section~\ref{sec:4}, we provide a numerical example of the kernel approximation. 
In Section~\ref{sec:5}, we apply the approximation result to an aggregation--confinement--diffusion equation. 
Finally, concluding remarks are given in Section~\ref{sec:6}.

\bigskip
\section{Approximation theorem}\label{sec:2}

We first state our main result for the approximation of integral kernels.
\begin{theorem}\label{thm:ker}
Let $\alpha$ be a non-negative integer
and $\{d_m\}_{m\in\bN}$ be a sequence of distinct positive constants that has a positive accumulation point.
Assume that $K\in H^{\alpha}(\bR^n)$ is a radial function.
Then, for any $\eps>0$,
there exist $M\in\bN$ satisfying
\be\label{number}
M>\dfrac{2\alpha+n}{4}
\ee
and constants $\{a_m\}_{1\le m\le M}$ such that 
\beaa
K_M \in H^{\alpha}(\bR^n)
\eeaa
and
\beaa
\left\| K - K_M \right\|_{H^{\alpha}(\bR^n)} < \eps.
\eeaa
\end{theorem}

Theorem~\ref{thm:ker} shows that every radial integral kernel in $H^\alpha(\bR^n)$ can be approximated arbitrarily well in the $H^\alpha$-norm by a linear sum of Green functions.
Because of its singularity at the origin, an individual Green function $k(\cdot;d)$ belongs to $L^2(\bR^n)$ only when $n\leq3$.
Nevertheless, the singularities of individual Green functions can be canceled by taking a suitable linear sum, thereby producing an approximating function that belongs to $H^\alpha(\bR^n)$.
The lower bound \eqref{number} reflects the number of Green functions needed to cancel their singularities sufficiently so that the resulting linear sum belongs to $H^{\alpha}(\bR^n)$.
The coefficients depend on both the integral kernel $K$ and the chosen number of terms $M$. 
At present, no explicit estimate relating the approximation error to $M$ is available, and obtaining such an estimate remains an open problem.

The proof is based on establishing the density of the linear span generated by the Green functions.
To prove this density, we apply the identity theorem for analytic functions.
This argument requires the assumption that the sequence $\{d_m\}_{m\in\bN}$ has an accumulation point in $(0,\infty)$.
The proof is given in Section~\ref{subsec:ker}.

We next discuss approximations in norms other than the one appearing in Theorem~\ref{thm:ker}.
For a nonnegative integer $l$ satisfying $\alpha>n/2+l$, the Sobolev--Morrey embedding yields
\beaa
H^{\alpha}(\bR^n)\subset C^{l}(\bR^n).
\eeaa
Consequently, if $K\in H^{\alpha}(\bR^n)$ with $\alpha\in\bN$ and $\alpha>n/2+l$, then Theorem~\ref{thm:ker} also yields an approximation in the $C^l$-norm.
More generally, for any nonnegative real number $\alpha$, an $H^\alpha$-approximation is obtained for kernels belonging to $H^{\lceil\alpha\rceil}(\bR^n)$, where $\lceil \cdot \rceil$ is the ceiling function.
Sobolev embedding also yields approximation in $L^p(\bR^n)$ for $p> 2$. 
On the other hand, the present Sobolev-space argument does not directly yield approximation in $L^p(\bR^n)$ for $1\le p<2$. 
We note that $L^1$-approximation of radial kernels has been established in \cite{ITpre} by a different approach.

The following corollary provides an approximation result for spatial convolutions by PDE systems.
\begin{corollary}\label{cor:con}
    Let $\{d_m\}_{m\in\bN}$ be a sequence of distinct positive constants that has a positive accumulation point.
    Let $\beta\in [0,n/2)$ and $\alpha\in\bN\cup\{0\}$ satisfy $\alpha\ge \beta$, and fix
    \beaa
        p\in\left[\frac{2n}{n-2\beta},\infty\right]
    \eeaa
    and $q\in[1,\infty]$ satisfying
    \beaa
        \dfrac{1}{q}=\dfrac{1}{p}+\dfrac{1}{2}+\dfrac{\beta}{n}.
    \eeaa
    When $K\in H^{\alpha}(\bR^n)$ is a radial integral kernel, for any $\eps>0$, there exist $M\in\bN$ and constants $\{a_m\}_{1\le m\le M}$ such that 
    \beaa
    K_{M}*f \in L^p(\bR^n)
    \eeaa
    and
    \beaa
    \left\| K* f - K_M*f \right\|_{L^p(\bR^n)} \le \eps \|f\|_{L^q(\bR^n)}
    \eeaa
    hold for every $f\in L^{q}(\bR^n)$.
\end{corollary}

This result shows that spatial convolutions can be approximated by the corresponding PDE systems in $L^p(\bR^n)$. 
The exponents are chosen so that Young's convolution inequality can be combined with the Sobolev embedding
\beaa
    H^\beta(\bR^n) \hookrightarrow L^{2n/(n-2\beta)}(\bR^n).
\eeaa
Thus, Corollary~\ref{cor:con} gives an operator-norm approximation of the convolution operator $f \mapsto K*f$ from $L^q(\bR^n)$ to $L^p(\bR^n)$.

The proof is based on Sobolev embedding and Young’s inequality and is given in Section \ref{subsec:con}.

\bigskip
\section{Completeness of a linear span of the Green functions}\label{sec:3}
In this section, we prove the results in Section~\ref{sec:2}. 
Throughout this section, let $\{d_m\}_{m\in\mathbb{N}}$ be a sequence of distinct positive constants having
a positive accumulation point, and let $\alpha$ be a non-negative integer.

We begin in Section~\ref{subsec:green} by collecting basic properties of the Green functions used throughout this paper. 
In Section~\ref{subsec:basis}, we introduce an auxiliary system $\{\phi_m\}_{m\in\bN}$ generated by finite linear sums of these Green functions $\{k_m\}_{m\in\bN}$ and establish its regularity in $H^\alpha(\bR^n)$.
In Section~\ref{subsec:comp}, we prove the density of its linear span. 
The proof consists of an identity-theorem argument for the analytic function $\cK$ in Lemma~\ref{lem:ide}, followed by a Fourier-transform uniqueness argument after a logarithmic change of variables in Lemma~\ref{lem:zero}. 
Finally, these results are combined in Sections~\ref{subsec:ker} and~\ref{subsec:con} to prove Theorem~\ref{thm:ker} and Corollary~\ref{cor:con}, respectively.

\subsection{Properties of the Green function}\label{subsec:green}
We introduce some properties of the Green function. 
It is known that $k(x;d)$ is represented as
\bea
k(x;d) &:=& \left(\dfrac{1}{d}\right)^{n/2} G\left( \dfrac{|x|}{\sqrt{d}} \right), \notag\\
G(|x|) &:=& \dfrac{1}{(2\pi)^{n/2}}\left( \dfrac{1}{|x|}\right)^{n/2-1} M_{n/2-1}\left(|x|\right), \label{green}
\eea
where $M_{\nu}(r)$ is the modified Bessel function of the second kind with order $\nu$ defined as
\beaa
M_{\nu}(r) := \int^{+\infty}_{0} e^{-r \cosh s} \cosh(\nu s)ds.
\eeaa

The Fourier transform of $G(|x|)$ defined by \eqref{green} is computed as
\beaa
\hat{G}(\xi) = \mathcal{F}_n[G](\xi) &:=& \int_{\bR^n} e^{-ix\cdot \xi} G(|x|)dx \\
&=&  (2\pi)^{n/2} \left( \dfrac{1}{|\xi|}\right)^{n/2-1} \int^{+\infty}_{0} J_{n/2-1}(|\xi|r) G(r)r^{n/2}dr \\
&=&\left( \dfrac{1}{|\xi|}\right)^{n/2-1} \int^{+\infty}_{0} r J_{n/2-1}(|\xi|r) M_{n/2-1}\left(r\right)dr
\eeaa
for $|\xi|>0$,
where $J_{n/2-1}(r)$ is the Bessel function with order $n/2-1$.
As it is known from \cite[\S 13.45]{Watson} that
\beaa
\int^{+\infty}_{0} r J_{n/2-1}(|\xi|r) M_{n/2-1}\left(r\right)dr= |\xi|^{n/2-1} {}_2 F_1 \left( \dfrac{n}{2}, 1; \dfrac{n}{2} ; - |\xi|^2 \right) = \dfrac{|\xi|^{n/2-1}}{1+|\xi|^2},
\eeaa
we have
$\hat{G}(\xi) = 1/(1+|\xi|^2)$.
Here, ${}_2F_1(a,b; c ; z)$ is the hypergeometric function.

It is easy to see that $G(|x|)$ is positive for $|x|>0$.
Using the integral formula
\beaa
\int^{+\infty}_{0} r^{\mu-1} M_{\nu}(r)dr = 2^{\mu-2} \Gamma\left(\dfrac{\mu-\nu}{2}\right) \Gamma\left(\dfrac{\mu+\nu}{2}\right)
\eeaa
with $|\mathrm{Re}(\nu)|< \mathrm{Re}(\mu)$ from \cite[\S 13.21]{Watson},
we obtain 
\beaa
\int_{\bR^{n}} G(|x|) dx &=& \dfrac{2\pi^{n/2}}{\Gamma(n/2)} \int^{+\infty}_{0} r^{n-1} G(r) dr \\
&=& \dfrac{1}{2^{n/2-1}\Gamma(n/2)} \int^{+\infty}_{0} r^{n/2} M_{n/2-1}(r) dr=1.
\eeaa
Moreover, from \cite[\S 7.10]{Tit} and the Legendre duplication formula, we have
\beaa
\int_{\bR^n} G(|x|)^2 dx &=& \dfrac{2\pi^{n/2}}{\Gamma(n/2)} \int^{+\infty}_{0} r^{n-1} G(r)^2 dr \\
&=& \dfrac{1}{2^{n-1}\pi^{n/2}\Gamma(n/2)} \int^{+\infty}_{0} r M_{n/2-1}(r)^2 dr \\
&=& \dfrac{1}{2^{n+1} \pi^{(n-1)/2}\Gamma(3/2)} \Gamma\left(2-\dfrac{n}{2}\right)
\eeaa
when $n\le 3$.

From these properties of $G(|x|)$, 
we obtain the following properties for $k_m$.
\begin{lemma}\label{lem:green}
    For $m\in\bN$, we have
    \begin{itemize}
        \item[(i)] $k_m\in C(\bR^n\backslash\{0\})\cap L^1(\bR^n)$ and
        \beaa
        \mathcal{F}_n[k_m](\xi)= \dfrac{1}{1+d_m |\xi|^2};
        \eeaa
        \item[(ii)] $k_m(x)>0\ (x\neq 0)$ and $\|k_m\|_{L^1(\bR^n)}=1$;
        \item[(iii)] $k_m \in L^2(\bR^n)$ and 
        \beaa
        \| k_m \|^2_{L^2(\bR^n)} = d^{-n/2}_m \| G\|^2_{L^2(\bR^n)}
        \eeaa
        when $n\le 3$.
    \end{itemize}
\end{lemma}

\medskip
\subsection{Auxiliary system generated by the Green functions}\label{subsec:basis}
In this subsection, we construct an auxiliary system of functions in $H^\alpha(\bR^n)$ generated by finite linear sums of the Green functions. 
Since the Green functions $\{k_m\}_{m\in\bN}$ do not necessarily belong to $H^\alpha(\bR^n)$, they are not suitable for a direct density argument. 
The auxiliary system $\{\phi_m\}_{m\in\bN}$ introduced below resolves this difficulty.

To consider only radial functions, we identify $K(x)$ with $K(|x|)$.
We denote the set of all radial functions in $H^{\alpha}(\bR^n)$ by $H^\alpha_{r}(\bR^n)$.
Here, we use the following equivalent Sobolev norm:
\begin{align*} 
\| K \|^2_{H^\alpha_r(\bR^n)}:= \int^{+\infty}_{0} s^{n-1} (1+s^2)^{\alpha} |\hat{K}(s)|^2 \, ds,
\end{align*}
where $\hat{K}(s)$ denotes the Fourier transform of $K(|x|)$.
This norm differs from the standard Sobolev norm only by a positive dimensional constant, which is immaterial in the arguments below.
The space of radial functions in $L^2(\bR^n)$ is defined by $L^2_r(\bR^n) := H^{0}_r(\bR^n)$.

The inner product for $H^{\alpha}_r(\bR^n)$ is defined by
\beaa
\lng K_1,K_2 \rng_{H^\alpha_r(\bR^n)} := \int^{+\infty}_{0} s^{n-1} (1+s^2)^\alpha \hat{K_1}(s) \hat{K_2}(s) ds\quad (K_1,K_2\in H^\alpha_r(\bR^n)).
\eeaa
In this section, we adopt the notation $r=|x|$ and $s=|\xi|$, and $\hat{\cdot}$ denotes the Fourier transform.

We begin by proving the linear independence of $\{k_m\}_{m\in\bN}$.
\begin{lemma}\label{lem:ind}
$\{k_m\}_{m\in\bN}$ is linearly independent.
\end{lemma}
\begin{proof}
For $M\in\bN$ and $\{a_m\}_{1\le m \le M}\subset\bR$, we assume that 
\beaa
\dsum^{M}_{m=1} a_m k_m(r) =0
\eeaa
holds for all $r>0$.
By taking the Fourier transform and using Lemma \ref{lem:green} (i),
we obtain
\beaa
\dsum^{M}_{m=1} \dfrac{a_m}{1+ d_m s^2} = 0
\eeaa
for all $s\ge 0$.
Multiplying by $\dfrac{1}{1+d_l s^2}\ (1\le l \le M)$ and integrating over $[0,+\infty)$ leads to the following system of linear equations
\beaa
\dsum^{M}_{m=1} \dfrac{a_m}{\sqrt{d_m}+ \sqrt{d_l} } = 0\quad (1\le l \le M),
\eeaa
because we have
\beaa
\int^{+\infty}_{0} \dfrac{ds}{(1+d_m s^2)(1+d_l s^2)} = \dfrac{\pi}{2(\sqrt{d_m}+\sqrt{d_l})}\quad (1\le m,l\le M).
\eeaa
Here, the matrix $\{1/(\sqrt{d_m}+ \sqrt{d_l}) \}_{1\le m,l\le M}$ is a Cauchy matrix, 
and the determinant can thus be computed according to
\beaa
&&\det \left( \left\{ \dfrac{1}{\sqrt{d_m}+ \sqrt{d_l}} \right\}_{1\le m,l\le M} \right) \\
&&\hspace{3cm} = 
\left( \prod^{M}_{m,l=1} (\sqrt{d_m}+\sqrt{d_l}) \right)^{-1} \left( \prod_{1\le m < l \le M} (\sqrt{d_l} - \sqrt{d_m})^2 \right)\neq 0.
\eeaa
That is to say, the matrix is invertible.
Therefore, $a_m =0$ holds for all $m=1,\ldots,M$.
\end{proof}

\medskip
To construct the auxiliary system, we fix a non-negative integer $L$ satisfying
\be\label{condiJ}
4L > 2\alpha + n -1
\ee 
and define two functions as
\beaa
\hat{w}(s) := \prod^{L}_{l=1} \dfrac{1}{1+ d_l s^2},
\quad
\hat{\phi_m}(s) :=  \hat{w}(s) \dfrac{1}{1+d_{m+L}s^2}
\eeaa
in the Fourier space.
It is easy to see that $\phi_m$ is represented by
\be\label{phi}
\phi_m(r) = \gamma_0 k_{m+L}(r) + \dsum^{L}_{l=1} \gamma_l k_l(r),
\ee
where $\gamma_{l}\ (0\le l \le L)$ are non-zero constants given by
\beaa
\gamma_0 = \prod^{L}_{l=1} \left(1 - \dfrac{d_l}{d_{m+L}}\right)^{-1},\quad \gamma_{l} =  \left(1 - \dfrac{d_{m+L}}{d_{l}}\right)^{-1} \prod_{1\le l^{*}\le L,\ l^{*}\neq l} \left(1 - \dfrac{d_{l^{*}}}{d_{l}}\right)^{-1}.
\eeaa
Hence, each $\phi_m$ is a finite linear sum of Green functions. 
The following lemma shows that the auxiliary system $\{\phi_m\}_{m\in\bN}$ is a linearly independent family in $H_r^\alpha(\bR^n)$.

\begin{lemma}
$\{\phi_m\}_{m\in\bN}$ is linearly independent. Moreover, $\{\phi_m\}_{m\in\bN} \subset H^{\alpha}_r(\bR^n)$ holds.
\end{lemma}
\begin{proof}
The linear independence of $\{\phi_m\}_{m\in\bN}$ follows from the linear independence of $\{k_m\}_{m\in\bN}$.
Fix $m\in\bN$. 
Let $d^{*}_m:= \min\{d_1,\ldots,d_L,d_{m+L}\}$.
As we have
\beaa
0\le \hat{\phi_m}(s) \le \left( \dfrac{1}{1+d^{*}_m s^2} \right)^{L+1}
\eeaa
for any $s\ge 0$,
we obtain
\beaa
\|\phi_m\|^2_{H^{\alpha}_r(\bR^n)} 
\le \int^{+\infty}_{0} s^{n-1}(1 +s^2)^{\alpha} (1+d^{*}_m s^2)^{-2(L+1)} ds < + \infty
\eeaa
from \eqref{condiJ}.
Thus, the desired assertions are obtained.
\end{proof}

\medskip

Therefore, $\{\phi_m\}_{m\in\bN}$ is the desired auxiliary system in $H^\alpha_r(\bR^n)$.
We next prove its completeness.
We use the following characterization to prove the density.


\begin{lemma}\label{lem:comp}
Assume that $K=0$ whenever $K\in H^\alpha_r(\bR^n)$ satisfies
\beaa
\lng K,\phi_m\rng_{H^\alpha_r(\bR^n)}=0
\qquad
(m\in\bN).
\eeaa
Then the linear span of $\{\phi_m\}_{m\in\bN}$ is dense in $H^\alpha_r(\bR^n)$.
\end{lemma}

\medskip
\subsection{Completeness of the auxiliary system}\label{subsec:comp}

Suppose that $K\in H^{\alpha}_r(\bR^n)$ satisfies $\lng K, \phi_m\rng_{H^\alpha_r(\bR^n)} =0$ for all $m\in\bN$.
In this subsection, we show that $K=0$ in $H^\alpha_r(\bR^n)$.

To prove that $K=0$, we first convert the condition
$\lng K,\phi_m\rng_{H^\alpha_r(\bR^n)}=0$
into a vanishing property of an analytic function.
For this purpose,
define 
\beaa
\hat{K}_{w}(s):= s^{n-1}(1+s^2)^{\alpha} \hat{w}(s) \hat{K}(s), \quad 
\mathcal{K}(z) := \int^{+\infty}_{0} \dfrac{1}{1+zs^2} \hat{K}_w(s)ds\quad (z\in\bC,\ \mathrm{Re}(z)>0).
\eeaa
Then, although $K_w$ also depends on $\alpha$ and $n$, we have 
\beaa
\int^{+\infty}_{0}  |\hat{K}_w(s)|^2  ds \le \left( \sup_{s\ge 0} s^{n-1}(1+s^2)^{\alpha} |\hat{w}(s)|^2 \right) \|K \|^2_{H^\alpha_r(\bR^n)} < +\infty
\eeaa
from \eqref{condiJ}.
We see that 
\beaa
0= \lng K, \phi_m\rng_{H^\alpha_r(\bR^n)} = \int^{+\infty}_{0} \dfrac{1}{1+d_{m+L} s^2} \hat{K}_w(s)ds = \cK(d_{m+L}).
\eeaa
Moreover, $\mathcal{K}$ is well defined on $\{z\in\bC\ \mid\ \mathrm{Re}(z)>0\}$ because we find
\beaa
|\cK(z)| &\le& \int^{+\infty}_{0} \dfrac{1}{|1+zs^2|} |\hat{K}_w(s)| ds \\ 
&\le& \int^{+\infty}_{0} \dfrac{1}{1+ \mathrm{Re}(z) s^2} |\hat{K}_w(s)| ds
\le \left(\dfrac{\pi}{4 \sqrt{\mathrm{Re}(z)}} \right)^{1/2} \left( \int^{+\infty}_{0}  |\hat{K}_w(s)|^2  ds \right)^{1/2}.
\eeaa
We obtain the following equality from the identity theorem.
\begin{lemma}\label{lem:ide}
Suppose that $\lng K, \phi_m\rng_{H^\alpha_r(\bR^n)} =0$ holds for all $m\in\bN$.
We then have $\cK (z)\equiv 0$ on $\{z\in\bC \mid \mathrm{Re}(z)>0\}$.
\end{lemma}
\begin{proof}
For $z_1,z_2\in\{z\in\bC \mid \mathrm{Re}(z)>0\}$ with $z_1\neq z_2$,
we deduce
\beaa
\dfrac{\cK(z_1)-\cK(z_2)}{z_1-z_2} 
&=& \dfrac{1}{z_1-z_2}  \int^{+\infty}_{0}  \hat{K}_w(s) \left[ \dfrac{1}{1+z_1 s^2} - \dfrac{1}{1+z_2 s^2} \right]  ds \\
&=& - \int^{+\infty}_{0} \dfrac{ s^{2} \hat{K}_w(s) }{(1+z_1 s^2)(1+z_2 s^2)} ds.
\eeaa
Thus, for any $\mathrm{Re}(z)>0$, we obtain
\beaa
\dfrac{d\cK}{dz}(z) &=& - \int^{+\infty}_{0} \dfrac{  s^{2} \hat{K}_w(s) }{(1+z s^2)^2} ds.
\eeaa
Therefore, $\cK$ is analytic on $\{z\in\bC \mid \mathrm{Re}(z)>0\}$.
As $\{d_m\}_{m\in\bN}$ has a positive accumulation point,
$\cK(z)\equiv 0$ holds from the identity theorem.
\end{proof}

\medskip
We next show that the identity $\mathcal K\equiv0$ implies $K=0$.
To this end, we rewrite $\mathcal K$ as a convolution on $\bR$ through the logarithmic change of variables and apply the Fourier transform.

\begin{lemma}\label{lem:zero}
If $\cK(z)\equiv 0$ on $\{z\in\bC \mid \mathrm{Re}(z)>0\}$, then $K= 0$ holds.
\end{lemma}
\begin{proof}
Let $\tilde{K}(q):= e^{-q/2} \hat{K}_w(e^{-q})$.
We then deduce that
\be\label{inte}
\|\tilde{K}\|^{2}_{L^2(\bR)} =\int_{\bR}  |\tilde{K}(q)|^2 dq = \int^{+\infty}_{0} |\hat{K}_w(s)|^2 ds   < +\infty.
\ee
Thus, $\tilde{K}\in L^2(\bR)$ holds.

Let
\beaa
h(\lambda) := \dfrac{e^{\lambda/2}}{1+e^{2\lambda}}.
\eeaa
For $\lambda\in\bR$, we have
\beaa
0 = \cK(e^{2\lambda})e^{\lambda/2} = \int^{+\infty}_{0} \dfrac{\hat{K}_w(s) e^{\lambda/2}}{1+e^{2\lambda} s^2} ds 
= \int^{+\infty}_{-\infty} \dfrac{e^{(\lambda-q)/2} \tilde{K}(q)}{1+e^{2(\lambda-q)}} dq = (h*\tilde{K})(\lambda).
\eeaa
It follows from $h\in L^1(\bR)$ that
\beaa
0= \int_{\bR} (h*\tilde{K})(\lambda) e^{-i\eta\lambda}d\lambda = 
\left( \int_{\bR} h(\lambda) e^{-i\eta\lambda}d\lambda \right) \left( \int_{\bR} \tilde{K}(\lambda) e^{-i\eta\lambda}d\lambda \right) 
\eeaa
for a.e. $\eta\in\bR$.
Moreover, we have
\beaa
\int_{\bR} \tilde{K}(\lambda) e^{-i\eta\lambda}d\lambda =0
\eeaa
for a.e. $\eta\in\bR$, because we know that  
\beaa
\int_{\bR} h(\lambda) e^{-i\eta \lambda} d\lambda 
= \int^{+\infty}_{0} \dfrac{s^{-1/2-i\eta} } {1+s^2}ds 
= \dfrac{\pi}{2\sin\left( \dfrac{\pi}{4} \left( 1-2i\eta\right)  \right)} \neq 0.
\eeaa
From the Parseval identity, 
we have $\|\tilde{K}\|_{L^2(\bR)}=0$ and thus 
$\hat{K}(s)= 0$ for a.e. $s\ge 0$ from \eqref{inte}.
Therefore, we conclude that $K = 0$.
\end{proof}

\medskip
\subsection{Approximation of radial integral kernels}\label{subsec:ker}

From the preparations above, we give the proof of Theorem \ref{thm:ker}.
\begin{proof}[Proof of Theorem \ref{thm:ker}]
Let $\alpha$ be a non-negative integer and $\{d_m\}_{m\in\bN}$ be a sequence of distinct positive constants that has a positive accumulation point.
As the system $\{\phi_m\}_{m\in\bN}$ of $H^\alpha_r(\bR^n)$ is complete from Lemmata \ref{lem:comp}, \ref{lem:ide}, and \ref{lem:zero}, 
for any $K\in H^\alpha_r(\bR^n)$ and $\eps>0$, there exists $M_0\in\bN$ and $\{b_m\}_{1\le m \le M_0}$ such that
\beaa
\left\|K - \dsum^{M_0}_{m=1} b_m \phi_m \right\|_{H^\alpha(\bR^n)} < \eps.
\eeaa
Moreover, using \eqref{phi}, we construct $\{ a_m \}_{1\le m \le L+M_0}$ satisfying
\beaa
\dsum^{M_0}_{m=1} b_m \phi_m = \dsum^{L+M_0}_{m=1} a_m k_m= K_{L+M_0}.
\eeaa
Thus, we obtain $K_{L+M_0}\in H^\alpha_r(\bR^n)$ and
\beaa
\left\|K - K_{L+M_0} \right\|_{H^\alpha(\bR^n)} < \eps.
\eeaa
Hence, the proof is complete.
\end{proof}

\medskip
\subsection{Approximation of spatial convolutions}\label{subsec:con}

Finally, we give the proof of Corollary \ref{cor:con}.

\begin{proof}[Proof of Corollary \ref{cor:con}]
    Let $\beta\in [0,n/2)$ and $\alpha\in\bN\cup\{0\}$ satisfy $\alpha\ge \beta$.
    In addition, let $\{d_m\}_{m\in\bN}$ be a sequence of distinct positive constants that has a positive accumulation point.
    Fix $K\in H^\alpha(\bR^n)$.
    Then, there is $C=C(\beta,n)>0$ such that
    \beaa
    \|K\|_{L^{2n/(n-2\beta)}(\bR^n)} \le C \|K\|_{H^{\beta}(\bR^n)} \le C \|K\|_{H^\alpha(\bR^n)}
    \eeaa
    holds from the Sobolev embedding theory.
    From Theorem \ref{thm:ker}, there exist $M\in\bN$ and constants $\{a_m\}_{1\le m\le M}$ such that 
    \beaa
    K_M \in H^\alpha(\bR^n)
    \eeaa
    and
    \beaa
    \left\| K - K_M \right\|_{H^\alpha(\bR^n)} \le \dfrac{\eps}{C}
    \eeaa
    hold. 

    Let $p,q\in[1,+\infty]$ satisfy $1/q=1/p+1/2 +\beta/n$, and let $f\in L^{q}(\bR^n)$ be arbitrary.
    We then obtain
    \beaa
    \left\|K_{M}*f\right\|_{L^p(\bR^n)} \le
    \left\| K_M \right\|_{L^{2n/(n-2\beta)}(\bR^n)} \|f\|_{L^q(\bR^n)} < +\infty.
    \eeaa
    We thus have $K_M*f \in L^{p}(\bR^n)$.
    Similarly, we have
    \beaa
    \left\| K* f - K_M *f  \right\|_{L^p(\bR^n)} \le \left\| K - K_M  \right\|_{L^{2n/(n-2\beta)}} \|f\|_{L^q(\bR^n)}  \le \eps \|f\|_{L^{q}(\bR^n)}.
   \eeaa
   Hence, the proof is complete.
\end{proof}

\bigskip
\section{Numerical example of the approximation of an integral kernel}\label{sec:4}

In this section, we introduce a numerical example of the approximation.
We present here how to find a linear sum of the Green functions that minimizes the error for a given radial integral kernel.

The discussion in the general case is essentially the same, and therefore the results in the case of $L^2(\bR^n)$ for $n=1,2,3$ are presented here.
Let $K\in L^2_r(\bR^n)$, $M\in\bN$ and $\{d_m\}_{m\in\bN}$ be distinct positive constants.
We define 
\beaa
E(b_1,\ldots,b_M):=\left\| K - \dsum^{M}_{m=1} b_m k_m \right\|^2_{L^2_r(\bR^n)}
\eeaa
for $(b_1,\ldots,b_M)\in\bR^M$.
As we find that
\beaa
\dfrac{\partial E}{\partial b_l} (b_1,\ldots,b_M) &=& -2 \left( \lng K, k_l \rng_{L^2_r(\bR^n)} - \dsum^{M}_{m=1} b_m \lng k_m,k_l \rng_{L^2_r(\bR^n)} \right) 
\eeaa
for $l=1,2,\ldots,M$, the critical point of $E$ is the solution of a system of linear equations
\be\label{eq:critical}
\dsum^{M}_{m=1} b_m \lng k_m,k_l \rng_{L^2_r(\bR^n)} = \lng K, k_l \rng_{L^2_r(\bR^n)}  \quad (1\le l \le M).
\ee
To analyze the simultaneous equation, we define the matrix $B=\{B_{m,l}\}_{1\le m,l\le M}$ as
\beaa
B_{m,l}:= \lng k_m , k_l \rng_{L^2_r(\bR^n)}
\eeaa
for $m,l\in\bN$.

\begin{lemma}
    The symmetric matrix $B$ is positive-definite.
\end{lemma}
\begin{proof}
    For $(c_1,\ldots,c_M)\in\bR^M$, we have
    \beaa
    \dsum^{M}_{m=1} \dsum^{M}_{l=1} c_m c_l B_{m,l} = \dsum^{M}_{m=1} \dsum^{M}_{l=1} \lng c_m k_m, c_l k_l \rng_{L^2_r(\bR^n)} = \left\| \dsum^{M}_{m=1}  c_m k_m  \right\|^{2}_{L^2_r(\bR^n)} \ge 0.
    \eeaa
    Moreover, it equals $0$ if and only if $c_m=0$ for all $m=1,\ldots,M$ from Lemma \ref{lem:ind}.
    Thus, $B$ is positive-definite.
\end{proof}

\medskip
Therefore, $E$ has a unique critical point $\{a_{m} \}_{m=1,\ldots,M}$.
Moreover, it satisfies 
\beaa
E(a_{1},\ldots,a_{M}) = \min_{(b_1,\ldots,b_M)\in\bR^M} E(b_1,\ldots,b_M),
\eeaa
because the Hessian matrix of $E$ is equal to $2B$ for any $(b_1,\ldots,b_M)\in\bR^M$.

To solve \eqref{eq:critical}, we specifically compute $B$.
\begin{lemma}
    For $m,l\in\bN$, the following identities hold:
    \begin{itemize}
        \item[(i)] when $n=1$, we obtain
        \beaa
        B_{m,l} = \dfrac{\pi}{2(\sqrt{d_m}+\sqrt{d_l})};
        \eeaa

        \item[(ii)] when $n=2$, we obtain
        \beaa
        B_{m,l} = 
        \begin{cases}
        \dfrac{1}{2d_m} & m=l, \\
        \dfrac{1}{2(d_m - d_l)} \log \dfrac{d_m}{d_l} & m\neq l;
        \end{cases}
        \eeaa

        \item[(iii)] when $n=3$, we obtain
        \beaa
        B_{m,l} = \dfrac{\pi}{2\sqrt{d_m d_l}(\sqrt{d_m}+\sqrt{d_l})}.
        \eeaa
    \end{itemize}
\end{lemma}

\medskip
To present numerical examples, we use 
\beaa
K(r)= e^{-r^2/4},\quad \hat{K}(s) = (\sqrt{4\pi})^n e^{-s^2}
\eeaa
and 
\beaa
d_m = 1 + \sin(m-1)\quad (m\in\bN).
\eeaa
The graphs of $K(x)$ and $K_M(x)$ are shown in Fig. \ref{fig:app}.
The graphs are largely similar in all cases, except for the neighborhood of $x=0$ in the case $n=3$.
This is because the Green function has a singularity at its origin.
As for the coefficients, $\max_{1\le m\le M}|a_m|$ is large in all cases, on the order of $10^{7}$ or more for $M=10$.
Even in this simple example, the coefficients can be large.
Moreover, note that the coefficients are not necessarily positive in this case, although we used a Gaussian, which is a positive-valued and positive-definite function.
The class of kernels $K(x)$ for which the coefficients
$\{a_m\}_{1\le m\le M}$ have a constant sign remains unclear, and its characterization is an open problem.

\begin{figure}[bt] 
\begin{center}
	\includegraphics[width=15cm]{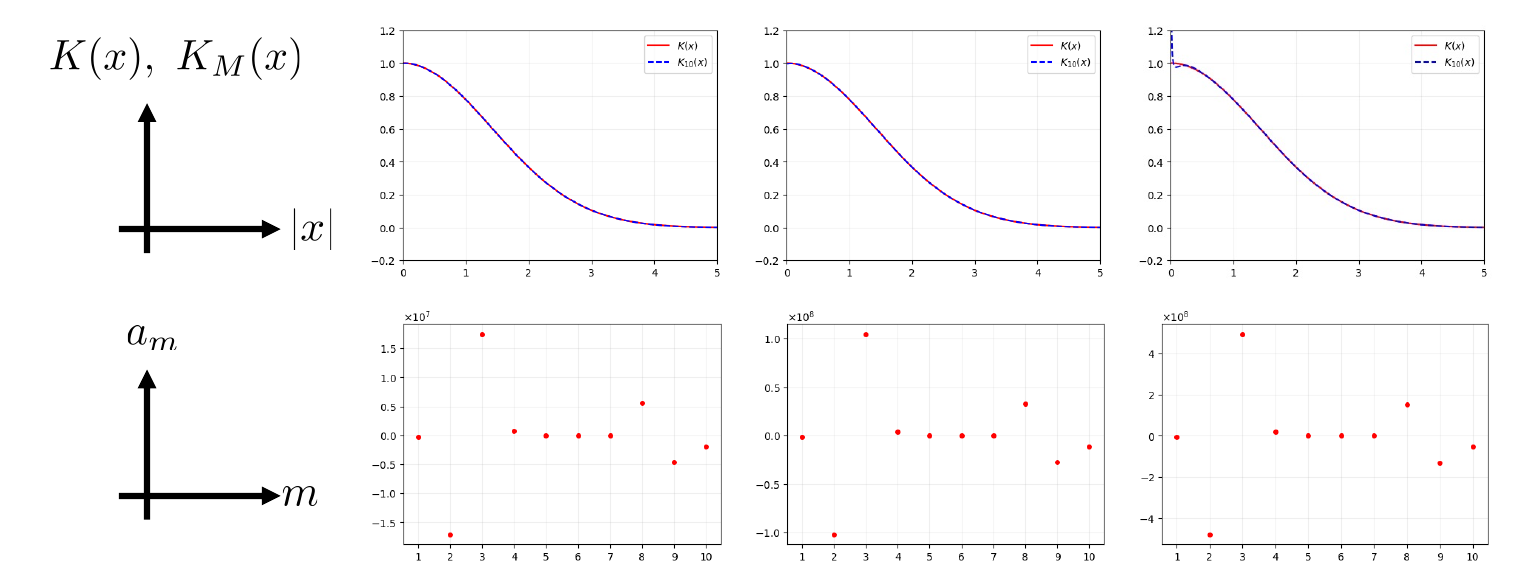}
\end{center}
\caption{\small{
Graphs of $K(x)$ and $K_M(x)$ with $M=10$ (top panels) and the distributions of $\{a_m\}_{1\le m\le M}$ (bottom panels).
The plots of $K(x)$ and $K_M(x)$ are shown by the red solid and blue dashed curves, respectively.
The numerical examples are presented for $n=1$ (left), $n=2$ (middle), and $n=3$ (right).
}}
\label{fig:app}
\end{figure}

We also show the approximation errors $\|K-K_M\|_{L^2(\bR^n)}$ against $M$ for $M=1,\ldots,40$ and $n=1,2,3$ in Fig.~\ref{fig:error}.
Here, the errors are measured in the standard $L^2(\bR^n)$ norm over the whole space.
The errors decrease rapidly as $M$ increases in all three cases over the range considered.

\begin{figure}[bt] 
\begin{center}
	\includegraphics[width=12cm]{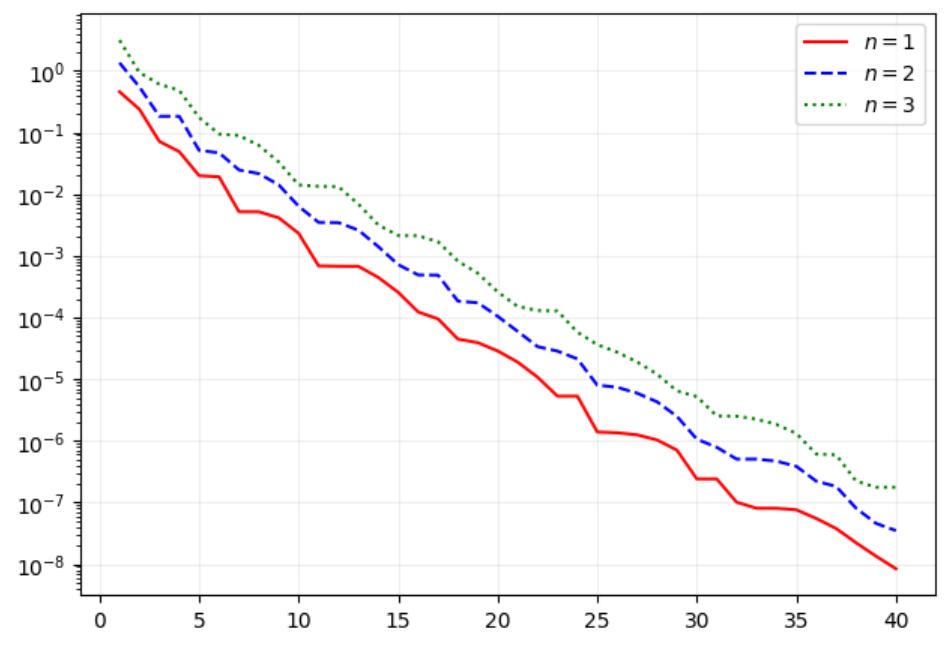}
\end{center}
\caption{\small{
     The approximation errors $\|K-K_M\|_{L^2(\bR^n)}$ against $M$ for $M=1,\ldots,40$, shown on a semilogarithmic scale.
    The red solid, blue dashed, and green dotted lines correspond to $n=1,2,3$, respectively.
}}
\label{fig:error}
\end{figure}

\begin{remark}
The determination of $\{a_m\}_{1\le m\le M}$ in other cases proceeds in a similar manner.
Note that the Green function belongs neither to $L^2(\bR^n)$ in the case of $n\ge 4$ nor to $H^\alpha(\bR^n)$ when $\alpha\ge 2$ and $n\ge 1$.
Thus, by replacing $\{k_m\}_{m\in\bN}$ with the auxiliary system $\{\phi_m\}_{m\in\bN}$, the coefficients can be determined through relation \eqref{phi}.
\end{remark}

\bigskip
\section{Application to the aggregation-confinement-diffusion equation}\label{sec:5}

We now illustrate how the kernel approximation theorem can be applied to a nonlinear evolution equation with nonlocal interaction.
Let us consider the Cauchy problem of the aggregation-confinement-diffusion equation
\be\label{pro:app}
\begin{cases}
    u_t = \Delta u - \nabla\cdot[u (\nabla V+ \nabla K*u)]\quad (t>0,\ x\in\bR^n), \\
    u(0,x) = u_0(x)\quad (x\in\bR^n),
\end{cases}
\ee
where $u=u(t,x)\in\bR\ (t>0,x\in\bR^n)$, the potential $V(x)$ satisfies $V \in W^{2,\infty}(\bR^n)$, 
the kernel $K$ is a radial function satisfying $ |\nabla K| \in L^\infty(\bR^n)$,
and the initial state $u_0\in L^1(\bR^n)$ is a non-negative function.
We note that for any $\alpha\in\bN$ with $\alpha>1+n/2$,
the condition $K\in H^\alpha(\bR^n)$ implies $|\nabla K| \in L^\infty(\bR^n)$.

The aggregation--confinement--diffusion equation can be derived as a mean-field limit of an interacting particle system subject to pairwise interactions and an external confining potential \cite{GC}. 
The corresponding free-energy functional is given by
\beaa
\cE(u) = \int_{\bR^n} u(x) \log u(x) dx - \int_{\bR^n} u(x) V(x) dx -  \dfrac{1}{2} \int_{\bR^n} u(x) (K*u)(x) dx.
\eeaa
The evolution of $u$ is formally characterized as the Wasserstein gradient flow of $\cE$, and therefore the free energy decreases along sufficiently regular solutions.

The aim of this section is to show that the solution to \eqref{pro:app} can be approximated by the following
parabolic-elliptic Keller-Segel system
\be\label{eq:KS}
\begin{cases}
    (u_M)_t = \Delta u_M - \nabla\cdot\left[u_M \left(\nabla V+ \dsum^M_{m=1} a_m \nabla v_m \right)\right]\quad (t>0,\ x\in\bR^n), \\
    0 = d_m \Delta v_m - v_m + u_M \quad (x\in\bR^n)
\end{cases}
\ee
with the initial condition
\be\label{init:KS}
u_M(0,x) = u_0(x)\quad (x\in\bR^n).
\ee
Here $v_m= k_m*u_M$, so that $v_m$ is the solution of the elliptic equation selected by convolution with the Green function $k_m$. 

Let us consider the following mild solution of the Cauchy problem \eqref{pro:app} given by
\be\label{sol:mild}
    u(t) = H(t)* u_0 - \int^{t}_{0} \nabla H(t-s)* [\, u(s) (\nabla V+ (\nabla K*u)(s))\,]ds =: \Phi[u](t),
\ee
where $H(t,x)$ denotes the heat kernel on $\bR^n$, namely
\beaa
H(t,x) := \dfrac{1}{(4\pi t)^{n/2}} \exp\left( - \dfrac{|x|^2}{4t} \right).
\eeaa

We now state the existence of the mild solution:
\begin{proposition}\label{prop:global}
    Cauchy problem \eqref{pro:app} admits a unique mild solution 
    \beaa
    u\in C([0,\infty);L^1(\bR^n))\cap C((0,\infty);W^{1,1}(\bR^n))
    \eeaa
    which satisfies the following properties:
    \begin{itemize}
        \item[(i)] for any $T>0$, there exists a positive constant $C_{D}= C_{D}(T)$ such that
        \beaa
        \sup_{t\in(0,T]} t^{1/2} \|\nabla u(t)\|_{L^1(\bR^n)} < C_{D}(T);
        \eeaa

    	\item[(ii)] the solution enjoys the regularity 
        \beaa
        u\in C^1((0,+\infty);L^1(\bR^n)),\qquad
        \Delta u \in C((0,+\infty);L^1(\bR^n)),
        \eeaa
        where the Laplacian is understood in the distributional sense. 
        Moreover, \eqref{pro:app} holds in $L^1(\bR^n)$ for every $t>0$;

        \item[(iii)] for every $t>0$, the function $u(t)$ is nonnegative on $\bR^n$, and it satisfies
        \be\label{mass}
        \int_{\bR^n} u(t,x)dx = \int_{\bR^n} u_0(x)dx.
        \ee
    \end{itemize}
\end{proposition}

\medskip
This result also applies to the Cauchy problem \eqref{eq:KS}--\eqref{init:KS}. 
The Cauchy problem corresponds to the case in the Cauchy problem~\eqref{pro:app} where $K = K_M$. 
Therefore, when $|\nabla K_M|\in L^\infty(\bR^n)$, the existence of a global solution is guaranteed by Proposition~\ref{prop:global}.

Note that the global existence of mild solutions to \eqref{pro:app} is known in \cite{KS} when $V\equiv 0$.
The proof of Proposition~\ref{prop:global} is based on a standard argument by the Banach fixed point theorem.
See Appendix~\ref{subsec:GE} for the proof.

By applying Theorem~\ref{thm:ker}, we obtain the following approximation result:
\begin{theorem}\label{thm:KS-app}
    Let $\alpha$ be a non-negative integer satisfying $\alpha>1+n/2$.
    Fix $T>0$ and a radial kernel $K\in H^{\alpha}(\bR^n)$.
    Let $\{d_m\}_{m\in\bN}$ be a sequence of distinct positive constants that has a positive accumulation point.
    Then, for any $\eps>0$, there exist $M\in\bN$ and constants $\{a_m\}_{1\le m\le M}$ such that
    \beaa
    \max_{0\le t\le T} \| u(t) - u_M(t) \|_{L^1(\bR^n)} \le \eps, 
    \eeaa
    where $u$ and $u_M$ are the solutions to the Cauchy problems~\eqref{pro:app} and \eqref{eq:KS}--\eqref{init:KS}, respectively.
\end{theorem}

\medskip
\subsection{Parabolic-elliptic Keller-Segel approximation}

In this subsection, we give the proof of Theorem~\ref{thm:KS-app}.
First, we establish an estimate for the stability of the Cauchy problem \eqref{pro:app} with respect to the integral kernel.
\begin{proposition}\label{prop:stability} 
    Suppose that $|\nabla K_M|\in L^\infty(\bR^n)$.
    Let $u$ and $u_M$ be the mild solutions to the Cauchy problems \eqref{pro:app} and \eqref{eq:KS}--\eqref{init:KS}, respectively.
    Then, for any $T>0$, we have
    \beaa
    \max_{t\in [0,T]}\|(u- u_M)(t)\|_{L^1(\bR^n)} \le C_1 \sqrt{T} E_{1/2} (C_2 \Gamma(1/2)\sqrt{T}) \|\nabla K-\nabla K_M\|_{L^{\infty}(\bR^n)}
    \eeaa
    where 
    \beaa
    C_1&:=& 2 C_H \|u_0\|^2_{L^1(\bR^n)},\\
    C_2&=& C_2(\|\nabla K\|_{L^\infty(\bR^n)}, \|\nabla K_M\|_{L^\infty(\bR^n)}) \\
    &:=& C_H \|u_0\|_{L^1(\bR^n)}(\|\nabla K\|_{L^\infty(\bR^n)}+\|\nabla K_M\|_{L^\infty(\bR^n)}) + C_H \|\nabla V\|_{L^\infty(\bR^n)},
    \eeaa 
    $E_{1/2}(z)$ is the Mittag-Leffler function with the order $1/2$, and $\Gamma(z)$ is the gamma function.
\end{proposition}

For this proof, we use a fractional Gronwall inequality. We state below the special case of \cite[Lemma~7.1.1]{Henry} that is needed here.
\begin{lemma}\label{lem:FG}
    Let $c_1,c_2\ge 0$ and $T\in(0,+\infty]$. Suppose that the nonnegative and locally integrable function $U(t)$ on $[0,T)$ satisfies
    \beaa
        U(t) \le c_1 + c_2 \int^{t}_{0} \dfrac{U(s)}{\sqrt{t-s}} ds
    \eeaa
    for $t\in[0,T)$.
    Then, it satisfies
    \beaa
        U(t) \le c_1  E_{1/2}(c_2 \Gamma(1/2) \sqrt{t})\quad (0\le t<T).
    \eeaa
\end{lemma}

\medskip
\begin{proof}[Proof of Proposition~\ref{prop:stability}]
    Under the condition $|\nabla K_M|\in L^\infty(\bR^n)$,
    there exists a unique mild solution 
    \beaa
        u_M\in C([0,T];L^1(\bR^n)) \cap C((0,T];W^{1,1}(\bR^n)) 
    \eeaa
    to the Cauchy problem~\eqref{eq:KS}--\eqref{init:KS} from Proposition~\ref{prop:global}.

    Let $U_M(t) := \|(u-u_M)(t)\|_{L^1(\bR^n)}$.
    By using the mild formulation~\eqref{sol:mild}, 
    Lemma~\ref{lem:heat}, Young's and H\"{o}lder's inequalities, and the mass conservation
    \beaa
        \| u(t) \|_{L^1(\bR^n)} = \| u_M(t) \|_{L^1(\bR^n)} = \| u_0\|_{L^1(\bR^n)},
    \eeaa    
    we obtain
    \beaa
        U_M(t)
        &\le&  \|\nabla V\|_{L^\infty(\bR^n)} \int^{t}_{0} \|\nabla H(t-s)\|_{L^1(\bR^n)} U_M(s) ds \\
        &&\quad + \int^{t}_{0} \|\nabla H (t-s)\|_{L^1(\bR^n)} \|u(s)\|_{L^1(\bR^n)} \|\nabla K*(u-u_M)(s)\|_{L^\infty(\bR^n)} ds \\
        &&\quad + \int^{t}_{0} \|\nabla H (t-s)\|_{L^1(\bR^n)} \|u(s)\|_{L^1(\bR^n)} \|\nabla (K-K_M)*u_M(s)\|_{L^\infty(\bR^n)} ds \\
        &&\quad + \int^{t}_{0} \|\nabla H (t-s)\|_{L^1(\bR^n)} U_M(s) \|(\nabla K_M*u_M)(s)\|_{L^\infty(\bR^n)} ds \\
        &\le& C_H  \|\nabla V\|_{L^\infty(\bR^n)} \int^{t}_{0} (t-s)^{-1/2} U_M(s) ds \\
        &&\quad + C_H \|\nabla K\|_{L^\infty(\bR^n)} \|u_0\|_{L^1(\bR^n)} \int^{t}_{0} (t-s)^{-1/2} U_M(s) ds \\
        &&\quad +C_H \|\nabla K- \nabla K_M\|_{L^\infty(\bR^n)} \|u_0\|^2_{L^1(\bR^n)} \int^{t}_{0}  (t-s)^{-1/2} ds \\
        &&\quad + C_H \|\nabla K_M\|_{L^\infty(\bR^n)} \|u_0\|_{L^1(\bR^n)} \int^{t}_{0} (t-s)^{-1/2} U_M(s) ds \\
        &\le& C_1 \sqrt{T} \|\nabla K- \nabla K_M\|_{L^\infty(\bR^n)} + C_2 \int^{t}_{0} (t-s)^{-1/2} U_M(s) ds.
    \eeaa
    From Lemma~\ref{lem:FG}, it follows that
    \beaa
    U_M(t) \le C_1 \sqrt{T} E_{1/2} (C_2 \Gamma(1/2)\sqrt{t}) \|\nabla K-\nabla K_M\|_{L^{\infty}(\bR^n)}
    \eeaa
    for any $t\in[0,T]$.
    The desired assertion is obtained from the monotonicity of $E_{1/2}(z)$.
\end{proof}

\medskip
Finally, let us show Theorem~\ref{thm:KS-app}.
For the proof, we use the following Sobolev embedding:
\begin{lemma}\label{lem:emb}
    Let $\alpha\in\bN$ satisfy $\alpha> 1+n/2$.
    Then, there exists $C_{ebd}>0$ such that
    for any $f\in H^{\alpha}(\bR^n)$, it holds that
    \beaa
        \|f\|_{W^{1,\infty}(\bR^n)} \le C_{ebd} \|f\|_{H^{\alpha}(\bR^n)}.
    \eeaa
\end{lemma}

\medskip
We give the proof of Theorem~\ref{thm:KS-app}.
\begin{proof}[Proof of Theorem~\ref{thm:KS-app}]
    Fix $T>0$, $\eps>0$ and $\alpha\in\bN$ with $\alpha> 1+n/2$.
    Let $\{d_m\}_{m\in\bN}$ be a sequence of distinct positive constants that has a positive accumulation point.
    Then, there exists a small $\eps_0>0$ such that
    \be\label{ieq:para}
         C_1 C_{ebd} \sqrt{T} E_{1/2} \bigl(C_2(\|\nabla K\|_{L^\infty(\bR^n)}, \|\nabla K\|_{L^\infty(\bR^n)}+C_{ebd} \eps_0) \Gamma(1/2)\sqrt{T}\bigr) \eps_0 \le \eps.
    \ee
    Moreover, there exist $M\in\bN$ and constants $\{a_m\}_{1\le m\le M}$ such that
    \beaa
      K_M\in H^{\alpha}(\bR^n),\quad   \| K - K_M  \|_{H^{\alpha}(\bR^n)} \le \eps_0
    \eeaa
    by Theorem~\ref{thm:ker}.
    From Lemma~\ref{lem:emb}, we obtain
    \beaa
        \|\nabla K\|_{L^\infty(\bR^n)} <+\infty
    \eeaa
    and  
    \be\label{ieq:KKM}
        \|\nabla K - \nabla K_M\|_{L^\infty(\bR^n)} \le C_{ebd} \eps_0.
    \ee
    Since 
    \be\label{ieq:KM}
        \|\nabla K_M\|_{L^\infty(\bR^n)} \le \|\nabla K\|_{L^\infty(\bR^n)} + C_{ebd} \eps_0,
    \ee
    the existence of the mild solution $u_M$ to the Cauchy problem~\eqref{eq:KS}--\eqref{init:KS} with these parameters is obtained from Proposition~\ref{prop:global}.

    Finally, by using Proposition~\ref{prop:stability}, \eqref{ieq:para}, \eqref{ieq:KKM} and \eqref{ieq:KM},
    we conclude that
    \beaa
        \max_{t\in [0,T]}\|(u- u_M)(t)\|_{L^1(\bR^n)} \le\eps.
    \eeaa
    Therefore, the proof is complete.
\end{proof}

\bigskip
\section{Concluding remarks}\label{sec:6}

This paper developed a method for approximating spatial convolutions by PDE systems in arbitrary spatial dimensions.
An application of our results is the approximation of solutions to evolution equations with spatial convolutions by time-evolution PDE systems, as proposed in \cite{MT, NTY1, NTY2}. 
The approximation method for an advection--diffusion equation with spatial convolution in the one-dimensional case was mainly established in \cite{MT} through the approximation of the derivative of the integral kernel. 
 Our results extend this approximation framework to higher-dimensional spaces.

In our recent work \cite{ITpre}, an $L^1$-approximation theorem for radial kernels was established in arbitrary spatial dimensions, and a constructive procedure for determining the approximation parameters was obtained for $n\le 3$. 
The present work complements this result by establishing approximation in Sobolev spaces, which provides control of derivatives of the approximating kernels. 
An important open problem is to obtain constructive error estimates in the $H^\alpha$ -approximation developed here.

\section*{Declaration of generative AI and AI-assisted technologies in the writing process}
During the preparation of this work, the authors used DeepL and Grammarly in order to improve their English writing. 
After using these tools/services, the authors reviewed and edited the content as needed and take full responsibility for the content of the publication.

\bigskip
\section*{Acknowledgments}
The authors express their sincere gratitude to Daisuke Kawagoe (Kyoto University) for fruitful discussions and Glenn Pennycook, MSc, from Edanz (https://jp.edanz.com/ac) for editing a draft of this manuscript. 
The authors were partially supported by JSPS KAKENHI Grant Number 24H00188. HI was partially supported by JSPS KAKENHI Grant Number 23K13013. YT was partially supported by JSPS KAKENHI Grant Numbers 22K03444 and 24K06848.

\appendix

\section{Global existence of the solution to \eqref{pro:app}}\label{subsec:GE}
\renewcommand{\theequation}{A.\arabic{equation}}

In this section, we provide the proof of Proposition~\ref{prop:global}.
For $T>0$, we introduce the Banach space 
\be\label{sol-space}
X(T) := \bigl\{ w\in C([0,T];L^1(\bR^n)) \cap C((0,T];W^{1,1}(\bR^n))\ \mid\ \|w\|_{X} < + \infty  \bigr\},
\ee
endowed with the norm
\beaa
\|w\|_{X(T)} := \max_{t\in [0,T]} \|w(t)\|_{L^1(\bR^n)} + \sup_{t\in (0,T]} t^{1/2}\|\nabla w(t)\|_{L^1(\bR^n)}.
\eeaa
To show the existence of a mild solution, we establish that there exists a fixed point $u$ of \eqref{sol:mild} on $X(T)$.

\medskip
We first investigate the property of the integral operator $\Phi$.
Hereafter, we often use the following lemma, which is easy to check.
\begin{lemma}\label{lem:heat}
    There exists a constant $C_{H}>0$ such that for all $t>0$,
    \beaa
    \|\nabla H(t)\|_{L^1} \le C_H t^{-1/2}.
    \eeaa
\end{lemma}

We first show that $\Phi:X(T)\to X(T)$ is well-defined.
\begin{lemma}
    For any $T>0$, $\Phi:X(T) \to X(T)$ is well-defined and bounded. Moreover, the estimate
    \be\label{p:est}
    \|\Phi(w)\|_{X(T)} \le  (1+C_H) \|u_0\|_{L^1(\bR^n)}
    + C_H\sqrt{T} \|w\|_{X(T)} B(\|w\|_{X(T)})
    \ee
    holds, where
    \beaa
    B(R)
    := 2(1+\pi) \|\nabla K\|_{L^\infty(\bR^n)} R +(2+\pi)\|\nabla V\|_{L^\infty(\bR^n)}
    +2 \sqrt{T} \|\Delta V\|_{L^\infty(\bR^n)}.
    \eeaa
\end{lemma}
\begin{proof}
    Let us show that $\Phi(w)\in X(T)$ for $w\in X(T)$.
    For simplicity, we set
    \beaa
    \Phi_1(t)&:=& H(t)*u_0 \\
    \Phi_2(w)(t) &:=& \int^{t}_{0} \nabla H(t-s)* [\, w(s) (\nabla V+ (\nabla K*w)(s))\,]ds.
    \eeaa
    It holds that $\Phi(w)= \Phi_1 - \Phi_2(w)$.

    We first show that
    \be\label{ieq:phi1}
    \|\Phi_1\|_{X(T)} \le (1 + C_H) \|u_0\|_{L^1(\bR^n)}. 
    \ee
    By Young's inequality, we have
    \beaa
    \|\Phi_1(t)\|_{L^1(\bR^n)} \le \|H(t)\|_{L^1(\bR^n)} \|u_0\|_{L^1(\bR^n)} = \|u_0\|_{L^1(\bR^n)}
    \eeaa
    for any $t>0$.
    To evaluate the gradient, using Lemma~\ref{lem:heat} and Young's inequality, we obtain
    \beaa
    \| \nabla \Phi_1(t)\|_{L^1(\bR^n)} \le \|\nabla H(t)\|_{L^1(\bR^n)} \|u_0\|_{L^1(\bR^n)} \le C_H t^{-1/2}  \|u_0\|_{L^1(\bR^n)}
    \eeaa
    for $t>0$, and thus 
    \beaa
    \sup_{t\in(0,T]} t^{1/2} \| \nabla \Phi_1(t)\|_{L^1(\bR^n)} \le C_H  \|u_0\|_{L^1(\bR^n)}.
    \eeaa
    Hence, \eqref{ieq:phi1} is obtained.

    Next, we evaluate $\Phi_2(w)$ for $w\in X(T)$.
    By Lemma~\ref{lem:heat}, Young's and H\"{o}lder's inequalities, we obtain
    \beaa
    &&\|\Phi_2(w)(t)\|_{L^1(\bR^n)} \\
    &&\quad \le \int^{t}_{0} \|\nabla H(t-s)\|_{L^1(\bR^n)} \|w(s)\|_{L^1(\bR^n)} \bigl( \|\nabla V\|_{L^\infty(\bR^n)}+ \|(\nabla K* w)(s)\|_{L^\infty(\bR^n)} \bigr) ds \\
    && \quad \le  C_H \int^{t}_{0} (t-s)^{-1/2} \|w(s)\|_{L^1(\bR^n)} \bigl( \|\nabla V\|_{L^\infty(\bR^n)}+ \|\nabla K\|_{L^\infty(\bR^n)} \| w(s)\|_{L^1(\bR^n)} \bigr) ds \\
    &&\quad \le  C_H \|w\|_{X(T)} \bigl(\|\nabla V\|_{L^\infty(\bR^n)}+ \|\nabla K\|_{L^\infty(\bR^n)} \|w\|_{X(T)} \bigr) \int^{t}_{0} (t-s)^{-1/2} ds \\
    &&\quad \le 2 C_H \sqrt{T} \|w\|_{X(T)} \bigl(\|\nabla V\|_{L^\infty(\bR^n)}+ \|\nabla K\|_{L^\infty(\bR^n)} \|w\|_{X(T)}\bigr).
    \eeaa

    Next, we consider the gradient of $\Phi_2(w)$.
    Since it follows that
    \beaa
    \nabla [\Phi_2(w)](t) &=& \int^{t}_{0}  \nabla H(t-s)* [\nabla w(s) \cdot (\nabla V + (\nabla K*w)(s))]ds \\
    &&\qquad + \int^{t}_{0} \nabla H(t-s)* [w(s) (\Delta V + (\nabla K* \nabla w)(s))]ds
    \eeaa
    we have
    \beaa
    && \| \nabla[\Phi_2(w)](t) \|_{L^1(\bR^n)} \\
    && \quad \le  \int^{t}_{0} \|\nabla H(t-s)\|_{L^1(\bR^n)} \| \nabla w(s)\|_{L^1(\bR^n)} \bigl(\|\nabla V\|_{L^\infty(\bR^n)}+ \|(\nabla K * w)(s)\|_{L^\infty(\bR^n)}\bigr) ds \\
    &&\quad\quad + \int^{t}_{0} \|\nabla H(t-s)\|_{L^1(\bR^n)} \|w(s)\|_{L^1(\bR^n)} \bigl(\|\Delta V\|_{L^\infty(\bR^n)}+ \|(\nabla K* \nabla w)(s)\|_{L^\infty(\bR^n)}\bigr) ds \\
    &&\quad \le C_H \|w\|_{X(T)} \bigl(\|\nabla V\|_{L^\infty(\bR^n)}+ \|\nabla K\|_{L^\infty(\bR^n)}  \| w\|_{X(T)} \bigr) \int^{t}_{0} (t-s)^{-1/2} s^{-1/2} ds \\
    &&\quad\quad + C_H \|w\|_{X(T)} \int^{t}_{0} (t-s)^{-1/2} (\|\Delta V\|_{L^\infty(\bR^n)}+ s^{-1/2}\|\nabla K\|_{L^\infty(\bR^n)} \|w\|_{X(T)}) ds \\
    &&\quad \le  \pi C_H \|w\|_{X(T)} (\|\nabla V\|_{L^\infty(\bR^n)}+ 2 \|\nabla K\|_{L^\infty(\bR^n)}  \|w\|_{X(T)}) +2 \sqrt{T} C_H \|w\|_{X(T)} \|\Delta V\|_{L^\infty(\bR^n)}
    \eeaa
    for $t\in(0,T]$. Thus, we conclude
    \beaa
    && \sup_{t\in(0,T]} t^{1/2}\| \nabla[\Phi_2(w)](t) \|_{L^1(\bR^n)} \\
    &&\quad \le C_H \sqrt{T} \|w\|_{X(T)}  \left\{ \pi \bigl(\|\nabla V\|_{L^\infty(\bR^n)}+ 2 \|\nabla K\|_{L^\infty(\bR^n)}  \|w\|_{X(T)}\bigr) +2 \sqrt{T} \|\Delta V\|_{L^\infty(\bR^n)} \right\}.
    \eeaa

    The continuity of $\Phi(w)$ in $L^1$ on $[0,T]$ and in $W^{1,1}$ on $(0,T]$ follows from the continuity of the heat semigroup in $L^1$, the dominated convergence theorem, and the above estimates.
    Consequently, $\Phi(w)\in X(T)$ and \eqref{p:est} hold. 
    This completes the proof. 
\end{proof}

\medskip
Next, we have the following estimate for $\Phi:X(T)\to X(T)$.
\begin{lemma}\label{lem:Lip}
    For any $T>0$ and $w_1,w_2\in X(T)$, we have
    \be\label{p:lip}
    \|\Phi(w_1) - \Phi(w_2)\|_{X(T)} \le C_H  \sqrt{T} C_{Lip}(\|w_1\|_{X(T)}, \|w_2\|_{X(T)}) \|w_1 - w_2\|_{X(T)},
    \ee
    where $C_{Lip}$ is defined as
    \beaa
    C_{Lip}(R_1, R_2) =  (2+\pi) \|\nabla V\|_{L^\infty(\bR^n)} + 2\sqrt{T} \|\Delta V\|_{L^\infty(\bR^n)} + 2(1+\pi)(R_1+R_2) \|\nabla K\|_{L^\infty(\bR^n)}.
    \eeaa
\end{lemma}
\begin{proof}
    Let $T>0$ and $w_1,w_2\in X(T)$.
    By Lemma~\ref{lem:heat} and Young's inequality, we get
    \beaa
    && \|\Phi(w_1)(t) - \Phi(w_2)(t)\|_{L^1(\bR^n)} \\
    &&\quad \le \|\nabla V\|_{L^\infty(\bR^n)} \int^{t}_{0} \|\nabla H(t-s)\|_{L^1(\bR^n)} \|(w_{1}-w_{2})(s)\|_{L^1(\bR^n)} ds \\
    &&\quad\quad  + \int^{t}_{0} \|\nabla H (t-s)\|_{L^1(\bR^n)} \|w_1(s) \nabla K*(w_{1}-w_2)(s)\|_{L^1(\bR^n)} ds \\
    &&\quad\quad  + \int^{t}_{0} \|\nabla H (t-s)\|_{L^1(\bR^n)} \|(w_1 - w_2)(s) (\nabla K*w_{2})(s)\|_{L^1(\bR^n)} ds \\
    &&\quad \le C_H \|\nabla V\|_{L^\infty(\bR^n)}  \|w_{1}-w_{2}\|_{X(T)} \int^{t}_{0} (t-s)^{-1/2}ds \\
    &&\quad\quad  + C_H \|\nabla K\|_{L^\infty(\bR^n)} \|w_1\|_{X(T)} \|w_{1}-w_{2}\|_{X(T)} \int^{t}_{0} (t-s)^{-1/2}ds \\
    &&\quad\quad  + C_H \|\nabla K\|_{L^\infty(\bR^n)} \|w_2\|_{X(T)} \|w_{1}-w_{2}\|_{X(T)} \int^{t}_{0} (t-s)^{-1/2}ds \\
    && \quad \le 2C_H \sqrt{T} \bigl\{ \|\nabla V\|_{L^\infty(\bR^n)} +  \|\nabla K\|_{L^\infty(\bR^n)} (\|w_1\|_{X(T)} + \|w_2\|_{X(T)})  \bigr\} \|w_{1}-w_{2}\|_{X(T)}.
    \eeaa

    Next, we estimate the gradient.
    Since the gradient of $\Phi(w_1)- \Phi(w_2)$ is represented by
    \beaa
        &&\nabla [\Phi(w_1) - \Phi(w_2)](t) \\
        && \quad = - \int^{t}_{0}  \nabla H(t-s)* [\nabla (w_1 - w_2)(s) \cdot \nabla V + (w_1 - w_2)(s) \Delta V]ds \\
        &&\quad\quad - \int^{t}_{0} \nabla H(t-s)* [ \nabla w_1(s) \cdot  (\nabla K* w_1)(s) - \nabla w_2(s) \cdot  (\nabla K* w_2)(s)]ds \\
        &&\quad\quad - \int^{t}_{0} \nabla H(t-s)* [ w_1(s) (\nabla K* \nabla w_1)(s) - w_2(s) (\nabla K* \nabla w_2)(s)]ds
    \eeaa
    we evaluate
    \beaa
    && \|\nabla (\Phi(w_1) -  \Phi(w_2))(t)\|_{L^1(\bR^n)} \\
    && \quad \le \|\nabla V\|_{L^\infty(\bR^n)} \int^{t}_{0} \|\nabla H(t-s)\|_{L^1(\bR^n)}   \|\nabla(w_{1}-w_{2})(s)\|_{L^1(\bR^n)} ds \\
    &&\quad\quad + \|\Delta V\|_{L^\infty(\bR^n)} \int^{t}_{0} \|\nabla H(t-s)\|_{L^1(\bR^n)}   \|(w_{1}-w_{2})(s)\|_{L^1(\bR^n)} ds \\
    &&\quad\quad +\int^{t}_{0} \|\nabla H(t-s)\|_{L^1(\bR^n)} \|\nabla (w_1-w_2)(s) \cdot  (\nabla K*w_1)(s)\|_{L^1(\bR^n)} ds \\
    &&\quad\quad +\int^{t}_{0} \|\nabla H(t-s)\|_{L^1(\bR^n)} \|\nabla w_2 (s) \cdot  (\nabla K*(w_1-w_2))(s)\|_{L^1(\bR^n)} ds \\
    &&\quad\quad +\int^{t}_{0} \|\nabla H(t-s)\|_{L^1(\bR^n)} \|(w_1 - w_2)(s) (\nabla K* \nabla w_1)(s)\|_{L^1(\bR^n)} ds \\
    &&\quad\quad +\int^{t}_{0} \|\nabla H(t-s)\|_{L^1(\bR^n)} \|w_2(s) (\nabla K* \nabla (w_1 -w_2))(s)\|_{L^1(\bR^n)} ds \\ 
    && \quad \le C_H \|\nabla V\|_{L^\infty(\bR^n)}   \|w_{1}-w_{2}\|_{X(T)} \int^{t}_{0} (t-s)^{-1/2} s^{-1/2}  ds \\
    &&\quad\quad +  C_H \|\Delta V\|_{L^\infty(\bR^n)} \|w_{1}-w_{2}\|_{X(T)} \int^{t}_{0} (t-s)^{-1/2} ds \\
    &&\quad\quad + 2 C_H \|\nabla K\|_{L^\infty(\bR^n)} (\|w_{1}\|_{X(T)} + \|w_{2}\|_{X(T)}) \|w_{1}-w_{2}\|_{X(T)} \int^{t}_{0} (t-s)^{-1/2} s^{-1/2} ds \\
    && \quad \le C_H  \bigl\{ \pi\|\nabla V\|_{L^\infty(\bR^n)} + 2\sqrt{T}\|\Delta V\|_{L^\infty(\bR^n)} \\
    &&\hspace{4cm} + 2\pi \|\nabla K\|_{L^\infty(\bR^n)} (\|w_{1}\|_{X(T)} + \|w_{2}\|_{X(T)}) \bigr\}\|w_{1}-w_{2}\|_{X(T)}.
    \eeaa
    Multiplying the last estimate by $t^{1/2}$ and taking the supremum over
    $t\in(0,T]$, we obtain
    \beaa
    &&\sup_{t\in(0,T]}t^{1/2}
    \|\nabla(\Phi(w_1)-\Phi(w_2))(t)\|_{L^1(\bR^n)}\\
    &&\quad \le
    C_H\sqrt{T}
    \Bigl\{
    \pi\|\nabla V\|_{L^\infty}
    +
    2\sqrt{T}\|\Delta V\|_{L^\infty}
    +
    2\pi\|\nabla K\|_{L^\infty}
    (\|w_1\|_{X(T)}+\|w_2\|_{X(T)})
    \Bigr\}
    \|w_1-w_2\|_{X(T)}.
    \eeaa
    Combining this with the above $L^1$ estimate, we obtain \eqref{p:lip}.

    Thus, the proof is complete.
\end{proof}

\medskip
We now establish the local existence of solutions.
\begin{lemma}\label{lem:local}
    Let $R>(1+C_H)\|u_0\|_{L^1(\bR^n)}$, and $T$ be a constant satisfying 
    \be\label{condi:T}
    \begin{cases}
    (1+C_H) \|u_0\|_{L^1(\bR^n)} + C_H\sqrt{T} R B(R) \le R, \\
    C_H \sqrt{T} C_{Lip}(R,R) <1
    \end{cases}
    \ee
    Then, the Cauchy problem \eqref{pro:app} admits a unique mild solution 
    \beaa
    u\in C([0,T];L^1(\bR^n)) \cap C((0,T];W^{1,1}(\bR^n))
    \eeaa
    which satisfies the following properties:
    \begin{itemize}
        \item[(i)]\ the solution has the upper bound
        \beaa
        \sup_{t\in(0,T]} t^{1/2} \|\nabla u(t)\|_{L^1(\bR^n)} \le R;
        \eeaa
        \item[(ii)]\ the solution enjoys the regularity 
        \beaa
        u\in C^1((0,T);L^1(\bR^n)),\qquad
        \Delta u \in C((0,T);L^1(\bR^n)),
        \eeaa
        where the Laplacian is understood in the distributional sense. 
        Moreover, \eqref{pro:app} holds in $L^1(\bR^n)$ for every $t\in(0,T)$;
        \item[(iii)]\ for each $t\in[0,T]$, the function $u(t)$ is nonnegative on $\bR^n$ and satisfies \eqref{mass}.
    \end{itemize}
\end{lemma}
\begin{proof}
    Fix $R>(1+C_H)\|u_0\|_{L^1(\bR^n)}$ and a positive constant $T$ satisfying \eqref{condi:T}.
    Let us define the closed subset 
    \beaa
    X_{R}(T) := \{u\in X(T) \mid \|u\|_{X(T)} \le R\}.
    \eeaa
    Then, $\Phi:X_{R}(T)\to X_R(T)$ is a contraction map from \eqref{p:est} and \eqref{p:lip}.
    Hence, Banach's fixed point theorem yields the existence of a fixed point $u\in X_R(T)$ of $u=\Phi(u)$.

    \bigskip
    \noindent
    \textbf{Uniqueness.}
    Let $w_1, w_2\in X(T)$ be two solutions of \eqref{sol:mild}.
    Then, we get
    \beaa
    &&\|(w_1- w_2)(t)\|_{L^1(\bR^n)} \\
    &&\quad\le \int^{t}_{0} \| \nabla H (t-s)* [ (w_1(s)- w_2(s)) \nabla V]\|_{L^1(\bR^n)}ds \\ 
    &&\quad\quad + \int^{t}_{0} \| \nabla H(t-s)* [w_1(s) (\nabla K*w_1)(s) - w_2(s) (\nabla K*w_2)(s)]\|_{L^1(\bR^n)}ds \\
    &&\quad \le C_H \|\nabla V\|_{L^{\infty}(\bR^n)} \int^{t}_{0} (t-s)^{-1/2}  \| (w_1- w_2)(s) \|_{L^1(\bR^n)} ds \\
    &&\quad\quad  + C_H \|\nabla K\|_{L^\infty(\bR^n)} (\|w_1\|_{X(T)} + \|w_2\|_{X(T)}) \int^{t}_{0} (t-s)^{-1/2}  \| (w_1- w_2)(s) \|_{L^1(\bR^n)}ds.
    \eeaa
    From Lemma~\ref{lem:FG}, it follows that
    \beaa
    \|(w_1 - w_2)(t)\|_{L^1(\R^n)} = 0
    \eeaa
    for any $t\in[0,T]$, and hence the solution is unique.

    \bigskip
    \noindent
    \textbf{Regularity.} Set 
    \be\label{nonlinear}
    f(t,x) := - \nabla \cdot [u(t,x) b(t,x)\,], \quad  b(t,x):=\nabla V(x)+(\nabla K* u)(t,x).
    \ee
    Let us first show that $f\in C([\tau,T];L^1(\bR^n))$ for any $\tau\in(0,T)$.
    Since $u\in C((0,T];W^{1,1}(\bR^n))$, 
    we have
    \beaa
        f(t)=-\nabla\cdot (u(t)b(t))=-\nabla u(t) \cdot b(t)-u(t) (\nabla \cdot b)(t),
    \eeaa
    where
    \beaa
        (\nabla \cdot b)(t)=\Delta V+ \nabla K* \nabla u(t).
    \eeaa
    Thus, for every $0<\tau<T$, we obtain
    \beaa
        b\in \{ C([\tau,T];W^{1,\infty}(\bR^n))\}^{n},\quad
        ub\in C([\tau,T];W^{1,1}(\bR^n))^{n}.
    \eeaa
    Hence, $f\in C([\tau,T];L^1(\bR^n))$.

    We next show that $f$ is locally H\"{o}lder continuous in time.
    It is enough to prove that, for every $\tau\in(0,T)$ and every
    $\theta\in (0,1/2)$,
    \beaa
        u\in C^\theta([\tau,T];W^{1,1}(\bR^n)).
    \eeaa
    Fix $0<\tau<T$ and set $\tau_0:=\tau/2$. 
    For $t\in[\tau,T]$, the mild formulation can be rewritten as
    \beaa
        u(t)=H(t-\tau_0)\ast u(\tau_0)
           +\int_{\tau_0}^{t} H(t-s)\ast f(s)\,ds . 
    \eeaa
    Since $f\in C([\tau_0,T];L^1(\bR^n))$, the standard smoothing estimates for the heat semigroup imply that
    \beaa
    u\in C^\theta([\tau,T];W^{1,1}(\bR^n))
    \qquad\text{for every } \theta\in(0,1/2).
    \eeaa
    Here, for $t,t+h\in[\tau,T],$ we used the following estimate
    \begin{align*}
    \| (H(t+h)-H(t))*f\|_{L^1(\R^n)}
    &=\left\| \left(\int_t^{t+h} \frac{d}{ds}H(s) ds\right)*f\right\|_{L^1(\R^n)}\\
    &=\left\| \left(\int_t^{t+h} \Delta H(s) ds\right)*f\right\|_{L^1(\R^n)}
    \le \frac{C_H h }{t}\|f \|_{L^1(\R^n)}.        
    \end{align*}
    From this, we have 
    \begin{equation*}
        \| \nabla(H(t+h) - H(t))*f\|_{L^1(\R^n)}\le C_H h t^{-3/2}\|f \|_{L^1(\R^n)}.
    \end{equation*}
    On the other hand, from the triangle inequality, 
    \begin{equation*}
        \| \nabla(H(t+h) - H(t))*f\|_{L^1(\R^n)}\le C_Ht^{-1/2} \|f \|_{L^1(\R^n)}.
    \end{equation*}
    Using a formula $\min\{ x,y\}\le x^{1-\theta}y^\theta$ for $0\le\theta\le1$, we have the interpolation as
    \begin{equation*}
        \| \nabla(H(t+h) - H(t))*f\|_{L^1(\R^n)}\le C_H h^\theta t^{-1/2-\theta}\|f \|_{L^1(\R^n)}.
    \end{equation*}
    Using this inequality, we obtain H\"{o}lder continuity in time of the Duhamel term with values in $W^{1,1}(\bR^n)$.
    This implies that
    \beaa
        f\in C^\theta([\tau,T];L^1(\bR^n)).
    \eeaa 

    Finally, we establish the time differentiability of $u$ and the $L^1$-regularity of its Laplacian.
    Let $A$ be the generator of the heat semigroup $S(t)$ on $L^1(\bR^n)$.
    For every $v\in D(A)$, we have $Av=\Delta v$ in the distributional sense.

    Fix $0<\sigma<T$.
    By the previous argument, $f\in C^\theta([\sigma,T];L^1(\bR^n))$ for any $0<\theta<1/2$.
    Consider the inhomogeneous problem
    \beaa
    \begin{cases}
        w'(t)=Aw(t)+f(t), & \sigma<t<T,\\
        w(\sigma)=u(\sigma).
    \end{cases}
    \eeaa
    By \cite[Theorem~4.1.7 and Remark~4.1.12]{Lorenzi}, its unique mild solution satisfies
    \beaa
        w\in C^1((\sigma,T);L^1(\bR^n))\cap C((\sigma,T];D(A)),
    \eeaa
    where $D(A)$ is equipped with the graph norm. 
    The variation-of-constants formula gives
    \beaa
        w(t)= S(t-\sigma)u(\sigma)+ \int^{t}_\sigma S(t-s)f(s)\,ds.
    \eeaa
    The shifted mild formulation for $u$ has the same right-hand side, and therefore $w=u$ on $[\sigma,T]$.

    Since $\sigma\in(0,T)$ was arbitrary, we conclude that
    \beaa
        u\in C^1((0,T);L^1(\bR^n)),\qquad \Delta u\in C((0,T];L^1(\bR^n)).
    \eeaa
    Moreover,
    \beaa
        u_t(t)=\Delta u(t)+f(t) = \Delta u(t)-\nabla\cdot(u(t)b(t))\quad\text{in }L^1(\bR^n)
    \eeaa
    for every $t\in(0,T)$.
    This proves assertion~(ii).

    \bigskip
    \noindent
    \textbf{Mass conservation.}
    Mass conservation \eqref{mass} follows from the mild formulation \eqref{sol:mild} and
    \beaa
    \int_{\bR^n} H(t,x)dx = 1,\quad \int_{\bR^n} \nabla H(t,x)dx = 0
    \eeaa
    for any $t>0$.

    \bigskip
    \noindent
    \textbf{Nonnegativity.}
    From the regularity result, we note that $u$ solves
    \be\label{eq:adv}
        u_t = \Delta u - \nabla \cdot (u b)
    \ee
    in $L^1(\bR^n)$ for $t\in(0,T)$, where $b$ is defined in \eqref{nonlinear}.
    Moreover,
    \beaa
        b\in L^\infty((0,T)\times\bR^n;\bR^n),\quad
        \nabla\cdot b \in L^1((0,T);L^\infty(\bR^n)).
    \eeaa

    Let $t_0\in (0,T]$ and let $\psi\in C^{\infty}_{c}(\bR^n)$ satisfy $\psi\ge 0$.
    Consider the backward adjoint problem
    \beaa
    \begin{cases}
        \Psi_{t} +\Delta \Psi + b(t,x) \cdot \nabla\Psi =0 &(t\in(0,t_0),\ x\in\bR^n), \\
        \Psi(t_0,x) = \psi(x).
    \end{cases}
    \eeaa   
    By a standard approximation argument, we first solve the adjoint problem with smooth bounded vector fields $b_j$ approximating $b$. 
    The maximum principle gives
    \beaa
        0\le \Psi_j(t,x)\le \|\psi\|_{L^\infty(\bR^n)}
        \qquad 0\le t\le t_0.
    \eeaa
    We choose the smooth approximations so that $b_j\to b$ almost everywhere and 
    \beaa
        \sup_j\|b_j\|_{L^\infty((0,t_0)\times\bR^n)}<\infty.
    \eeaa
    The Duhamel formula and Lemma~\ref{lem:FG}, applied after reversing time, yield
    \beaa
        \sup_j \|\nabla\Psi_j\|_{L^\infty((0,t_0)\times\bR^n)}<\infty.
    \eeaa
    Testing \eqref{eq:adv} with $\Psi_j$ using spatial cutoffs, and first integrating over $(\delta,t_0)$, we obtain, after removing the cutoffs and letting $\delta\downarrow0$,
    \beaa
        \int_{\bR^n}u(t_0,x)\psi(x)\,dx &=& \int_{\bR^n}u_0(x)\Psi_j(0,x)\,dx\\
        &&\qquad + \int^{t_0}_{0}\int_{\bR^n} u(t,x)(b-b_j)(t,x)\cdot\nabla\Psi_j(t,x)\,dx\,dt.
    \eeaa
    Here, the passage to the initial time follows from $u\in C([0,T];L^1(\mathbb{R}^n))$.
    The last integral tends to zero by dominated convergence, using the uniform bounds above and $u\in L^1((0,t_0)\times\bR^n)$.
    The initial-data term is nonnegative because $u_0\ge0$ and $\Psi_j(0)\ge0$.
    Thus, we have
    \beaa
        \int_{\bR^n} u(t_0,x)\psi(x)dx \ge 0
    \eeaa
    for every nonnegative $\psi\in C^{\infty}_{c}(\bR^n)$. Therefore,
    \beaa
        u(t_0,x) \ge 0
    \eeaa
    for a.e. $x\in\bR^n$. 
    Since $t_0\in(0,T)$ is arbitrary, we obtain the desired assertion.
\end{proof}

\medskip
Finally, we provide the proof of Proposition~\ref{prop:global}.
\begin{proof}[Proof of Proposition~\ref{prop:global}]
    Fix $R> (1+C_H)\|u_0\|_{L^1(\bR^n)}$ and $T>0$ satisfying \eqref{condi:T}.  
    By Lemma~\ref{lem:local}, the Cauchy problem \eqref{pro:app} admits a unique mild solution
    \beaa
    u \in C([0,T];L^{1}(\bR^n)) \cap C((0,T];W^{1,1}(\bR^n)).
    \eeaa
    From the nonnegativity and mass conservation established in Lemma~\ref{lem:local}, we have
    \beaa
    \|u(t)\|_{L^{1}(\bR^n)}
    = \int_{\bR^n} u_0 (x)\,dx = \|u_0\|_{L^1(\bR^n)},
    \qquad t \in [0,T].
    \eeaa
    Since this estimate is uniform in time and no longer depends on the technical smallness assumption on $T$, we may iterate the local existence result.  
    More precisely, for any $N\in\mathbb{N}$, the standard semigroup argument together with Lemma~\ref{lem:local} ensures the existence of a unique solution
    \beaa
    u \in C([0,NT];L^{1}(\bR^n)) \cap C((0,NT];W^{1,1}(\bR^n))
    \eeaa
    to the Cauchy problem \eqref{pro:app}.  
    Letting $N\to+\infty$, we obtain a unique global solution.

    The remaining assertions follow directly from Lemma~\ref{lem:local}.  
    This completes the proof.
\end{proof}


\end{document}